\documentclass{amsart}
\usepackage[utf8]{inputenc}
\usepackage{amsmath}
\usepackage{changepage}
\usepackage[backend=bibtex, sorting=none, bibstyle=alphabetic, citestyle=alphabetic, sorting=nyt,maxbibnames=99]{biblatex}
\usepackage{graphicx}
\usepackage{caption}
\usepackage{xcolor}
\usepackage{hyperref}
\usepackage{subfiles}
\usepackage{comment}
\usepackage{mathrsfs}
\usepackage{enumitem}
\usepackage{dsfont}

\newcommand{\var}{\mathrm{Var}}

\newtheorem{theorem}{Theorem}
\newtheorem{prop}[theorem]{Proposition}
\newtheorem{lemma}[theorem]{Lemma}
\newtheorem{cor}[theorem]{Corollary}
\theoremstyle{remark}
\newtheorem*{remark}{Remark}

\theoremstyle{definition}
\newtheorem{defn}{Definition}[section]

\bibliography{main.bib}

\title[Counting sign changes]{Counting sign changes of partial sums of random multiplicative functions}
\author{Nick Geis and Ghaith Hiary}

\begin{document}

\maketitle

\begin{abstract}
Let $f$ be a Rademacher random multiplicative function. Let
$$M_f(u):=\sum_{n \leq u} f(n)$$
be the partial sum of $f$.
Let $V_f(x)$ denote the number of sign changes of $M_f(u)$ up to $x$. 
We show that for any constant $c > 2$, 
$$V_f(x) = \Omega ((\log \log \log x)^{1/c} )$$ 
almost surely.
\end{abstract}

\section{Introduction}

Let $\{f(p)\}_{p \text{-prime}}$ be a set of independent and identically distributed random variables taking $\pm 1$ with probability $1/2$. 
Extend $f$ multiplicatively to squarefree numbers $n$ by
\[
f(n) = 
 \displaystyle   \prod_{p \mid n} f(p),
\]
and set $f(n)=0$ if $n$ is not squarefree.
Such $f$ is called a \emph{Rademacher random multiplicative variable}. These functions were introduced in 1944 by Wintner \cite{wintner1944} as a model for the M\"obius function $\mu(n)$. Wintner  studied the random Dirichlet series
\[
F(s) := \sum_{n \geq 1} \frac{f(n)}{n^s},
\]
showing that $F(s)$, as well as its Euler product
\[
F(s) = \prod_p \left( 1 + \frac{f(p)}{p^s} \right),
\] 
are convergent for $\Re(s) > 1/2$ almost surely, and that the function defined by them has a natural boundary on the line $\Re(s)=1/2$.
Consequently, Wintner derived an almost sure upper bound and almost sure $\Omega$-bound for the partial sum 
\[
M_f(x) = \sum_{n\leq x} f(n),
\]
of the form 
$$M_f(x)=O(x^{1/2+\epsilon})\qquad \text{and}\qquad M_f(x)\ne O_{\epsilon}(x^{1/2-\epsilon}),$$ 
for any fixed positive $\epsilon$.

The study of random multiplicative functions has since  been largely concerned with improving Wintner's bounds on $M_f(x)$, in pursuit of an analogue of the law of iterated logarithm. This includes the seminal work of Hal\'asz~\cite{halasz}. 

In 2023, the central question of the true size of the fluctuations of $M_f(x)$ is finally resolved if one combines the work of Harper \cite{harper2023almost} with the upcoming work of Caich \cite{caich2023}. Their combined efforts yield the following sharp upper bound on $M_f(x)$: for any $\epsilon > 0$, almost surely,
\[
M_f(x) = O_{\epsilon} \Big( \sqrt{x} (\log \log x)^{1/4 + \epsilon} \Big).
\] 
This upper bound is sharp in the sense that the $1/4$ exponent cannot be replaced by any smaller number.\footnote{The big-$O$ estimate applies when $x\ge x_0$ where $x_0$ may depend on $f$.} With this question resolved, other questions about the behavior of $M_f(x)$ remain of interest, such as the number of sign changes of $M_f(u)$ up to $x$ with $x$ tending to $\infty$, which is the subject of this work.

We say that a real-valued function $\varphi(x)$ has a \emph{sign change} in the interval $[a, b]$ if neither $\varphi(x) \geq 0$ nor $\varphi(x) \leq 0$ holds throughout $[a, b]$. Furthermore, we say that $\varphi(x)$ \emph{changes signs infinitely often as $x \to \infty$} if neither $\varphi(x) \geq 0$ nor $\varphi(x) \leq 0$ holds for all sufficiently large $x$. 
Aymone, Heap and Zhao \cite{sign_changes} observed that the methods of Harper \cite{harper2023almost} produce arbitrarily large positive and arbitrarily large negative values of $M_f(x)$ as $x\to \infty$, almost surely. Consequently, $M_f(x)$ changes signs infinitely often as $x \to \infty$, almost surely. 
\cite{sign_changes} found a particularly simple proof of this consequence of Harper's work, 
inspired by corresponding proofs in the deterministic setting for sums like $\sum_{n \leq x} \mu(n)$, where $\mu$ is the Mobius function.\footnote{See Chapter 11 of Bateman and Diamond \cite{batemandiamond} for an introduction to these methods.} The proof in \cite{sign_changes} 
relies, essentially, on the partial summation formula
\begin{equation}\label{partial summation}
F(s) =s\int_1^{\infty} \frac{M_f(u)}{u^{1+s}}\,du,
\end{equation}
where $s=\sigma+it$ and $\sigma>1/2$,
together with the following two facts. First, there exists a sequence $\sigma_{\ell} \to 1/2^+$ such that almost surely,
\begin{equation}\label{eq:facts1}
\int_1^{\infty} \frac{M_f(u)}{u^{1 + \sigma_{\ell}}} \ du \to 0
\end{equation}
as  $\ell \to \infty$.
Second, almost surely,
\begin{equation}\label{eq:facts2}
\int_1^{\infty} \frac{|M_f(u)|}{u^{1 + \sigma}} \, du \to \infty,
\end{equation}
as $\sigma\to 1/2^+$. The contrast between the behaviors of $M_f(x)$ in \eqref{eq:facts1} versus \eqref{eq:facts2} captures the necessarily infinitely-many sign changes of $M_f(x)$.   
The proof in \cite{sign_changes} does not, however, supply information about how often the sign changes of $M_f(x)$ occur, almost surely. 
Answering this question is the goal of this work. 

To quantify the main result in \cite{sign_changes} - that is, to quantify the number of sign changes - we first construct a sequence of growing intervals $[y_k, X_k]$ that eventually (for large enough $k$) each contain at least one sign change of $M_f(x)$, almost surely. This is the content of our main theorem.

\begin{theorem}\label{thm:mainresult}
    Let $c$, $A_0$, and $A_1$ be any constants satisfying $c > 2$, $A_0 \in (0, 1/6)$, and $A_1 > 1$. 
    Define the sequences
    \[
    \sigma_k := \frac{1}{2} + \frac{1}{\exp\left( \exp(k^c) \right)},
    \]
    \[
    X_k := \exp\left( \frac{1}{(\sigma_k - 1/2)^2} \right),
    \]
    and
    \[
    y_k := \exp \left( \left( \frac{1}{\sigma_k - 1/2} \right)^{A_0} \log^{-A_1} \left( \frac{1}{\sigma_k - 1/2} \right) \right).
    \]
    where $k$ runs through the positive integers.
Then, for almost all random multiplicative functions $f$ there exists a constant $k_0 =k_0(f,c,A_0,A_1) \geq 1$ such that $M_f(u)$ has at least one sign change in the intervals $[y_k, X_k]$ for each $k \geq k_0$.
\end{theorem}

As a consequence of Thereom~\ref{thm:mainresult}, we establish in Corollary~\ref{cor:1} an almost sure lower bound on the number of sign changes of $M_f(u)$ as $u$ ranges in the interval $[1,x]$ with $x\to \infty$. This is achieved by first choosing $A_0$ and $A_1$ so as to fit as many intervals $[y_k, X_k]$ in $[1, x]$ and then counting how many of those intervals are mutually disjoint. 

\begin{cor}\label{cor:1}
    Let $V_f(x)$ denote the number of sign-changes that $M_f(u)$ has in $[1, x]$. Let $c$ be any constant such that $c> 2$. For almost all random multiplicative functions $f$ there exists constants $x_0:= x_0(f, c) \geq 1$ and $C := C(f, c) \geq 1$ such that
    \[
    V_f(x) \geq \left(\log \left( 7 \log \log x \right)\right)^{1/c} - C
    \]
    for all $x \geq x_0$. In other words\footnote{ The $\Omega$-notation is in the sense of Hardy and Littlewood (Section 5.1 \cite{tenenbaum2015introduction}). Namely, we say that $f(x) = \Omega(g(x))$ if
    $\limsup_{x \to \infty} \left| f(x)/g(x) \right| > 0$.},
     \[
    V_f(x) = \Omega \left( (\log \log \log x)^{1/c} \right),
    \]
    almost surely.
\end{cor}

Our proof strategy is similar to \cite{sign_changes}, but makes adjustments to \eqref{eq:facts1} and \eqref{eq:facts2}. 
First, the assertion in \eqref{eq:facts1} expresses that the random Dirichlet series $F$ satisfies $F(\sigma_{\ell}) \to 0$ along some sequence $\sigma_{\ell} \to 1/2^+$, almost surely.  We improve this by showing the same holds along any sequence tending to $1/2^+$; that is,
    $F(\sigma) \to 0$ as $\sigma \to 1/2^+$, almost surely. 
    As this property of $F$ is frequently stated but with no certain reference, as far as we could tell,
    this seems to fill an apparent small gap in the literature.  To this end, we followed the suggestion 
    in \cite[Remark 2.1]{sign_changes} which requires the following key ingredient, proved in 
    \S{\ref{sec:thm2}}.
\begin{theorem}\label{thm:2}
    Let $\epsilon_1$ and $c_1$ be any constants satisfying $\epsilon_1 > 0$ and $c_1 > \sqrt{2}$.  
    For almost all $f$ there exists a constant $\sigma_0 > 1/2$, which may depend on $f$, such that
    \[
    \sum_{p} \frac{f(p)}{p^{\sigma}} \leq c_1 \left( \log\left( \frac{1}{2 \sigma - 1} \right) \right)^{1/2 + \epsilon_1}
    \]
    for all $\sigma \in (1/2 , \sigma_0).$ Therefore, 
    \[
    \sum_{p} \frac{f(p)}{p^{\sigma}} \ll  \left( \log\left( \frac{1}{2 \sigma - 1} \right) \right)^{1/2 + \epsilon_1}
    \]
    as $\sigma \to 1/2^+,$ almost surely.
\end{theorem}

\noindent
Second, the assertion in \eqref{eq:facts1} is used in \cite{sign_changes} to show that for any fixed real $t \neq 0$, 
\begin{equation}\label{Ft divergence}
\limsup_{\sigma \to 1/2^+} |F(\sigma + it)| = \infty,
\end{equation}
almost surely.  This, however, comes with no control over the rate of divergence, making \eqref{Ft divergence} of limited use in our context as we require more quantitative information. In fact, there appears little hope of gaining such control over the rate of divergence without a new approach. To circumvent this obstacle, we utilize a result of Harper on supremum of Gaussian processes \cite{harper2013}. This allows us to leverage the choice of $t$ to gain definitive information on the rate of divergence. Specifically,  
instead of considering lower bounds on $|F(\sigma + it)|$ for fixed $t\ne0$ as $\sigma\to 1/2^+$, we consider
\begin{equation}\label{eq:gauss}
\sup_{1 \leq t \leq T(\sigma)} |F(\sigma + it)|
\end{equation}
as $\sigma \to 1/2^+$, where $T$ is a real-valued function satisfying $T(\sigma) \to \infty$. By choosing an appropriate sequence $\sigma_k \to 1/2^+$ and taking 
$$T(\sigma) = 2\log^2\left( \frac{1}{\sigma - 1/2} \right),$$ 
we obtain a fairly subtle almost sure lower bound on the supremum \eqref{eq:gauss}. 
This gives us sufficient control over the rate of divergence, as detailed in \S{\ref{sec:harper}}.

Let us outline the proof of Theorem~\ref{thm:mainresult}. 
We first restrict to a subset of $f$'s satisfying certain desirable properties that we can show hold almost surely. 
In particular, the $f$'s in our subset satisfy the almost sure lower bound on Gaussian processes related to (\ref{eq:gauss}), as $\sigma\to 1/2^+$. We  denote this lower bound
by $L(\sigma)$. Also, the $f$'s in our subset all have corresponding random Dirichlet series $F$  
that satisfy $F(\sigma) \to 0$ as $\sigma \to 1/2^+$.

Following the notation in Theorem~\ref{thm:mainresult} we then assume, contrary to the assertion of the theorem,  
that $M_f(u)$ does not have a sign change on infinitely many 
intervals $[y_k, X_k]$.
While each $f$ in our subset may have a different infinite sequence $\mathcal{I}(f)$ of intervals $[y_k, X_k]$ where $M_f(u)$ has no sign change, it only matters that such a sequence exists for every $f$ in question. Using a certain ``flip trick'', our assumption to the contrary thus allows for an estimate of the form
\[
\int_1^{\infty} \frac{M_f(u)}{u^{1 + \sigma_k}} \ du \leq \int_1^{\infty} \frac{|M_f(u)|}{u^{1 + \sigma_k}} = O \Big( \log^3 y_k \Big),
\] 
valid for all $[y_k,X_k]$ in the sequence $\mathcal{I}(f)$, provided 
$k$ is sufficiently large. 

Put together, we therefore obtain the following chain of inequalities, valid for all $[y_k,X_k]$ in the sequence $\mathcal{I}(f)$ with $k$ sufficiently large,
\[
L(\sigma_k) \leq \sup_{1 \leq t \leq T(\sigma_k)} |F(\sigma_k + it)| \leq \int_1^{\infty} \frac{|M_f(u)|}{u^{1 + \sigma_k}} \ du \leq C_2 \log^3 y_k,
\]
where $C_2 > 0$ is an absolute constant. Finally, using an explicit form for $L(\sigma)$ supplied by \cite{harper2013}, we are able to arrive at a contradiction by taking $k$ even larger if necessary. Consequently, once $k$ is past a certain threshold $k_0$ (depending on $f$), each of the intervals $[y_k, X_k]$ must contain a sign change of $M_f(u)$.

We remark that our lower bound on the sign counting function $V_f(x)$ in Corollary \ref{cor:1} is probably far from the truth. 
This is motivated by two observations. First, Kaczorowski and Pintz~\cite{deterministic_case} showed that $\sum_{n \leq u} \mu(n)$ has at least $(\gamma_1/\pi - o(1))\log x$ sign changes in the interval $[1, x]$ as $x \to \infty$, where $\gamma_1=14.134725\ldots$ is the height of the first non-trivial zeta zero. If $M_f(x)$ is an accurate random  model for $\sum_{n\le x} \mu(n)$, then an almost sure lower bound on $V_f(x)$ of comparable quality can be expected to hold.
Second, let $\{X(n)\}_{n \in \mathbb{N}}$ be a family of independent and identically distributed random variables uniformly distributed on $\{\pm 1\}$. So, there is no multiplicative dependence. Let $S(u) = \sum_{n \leq u} X(n)$ be the corresponding partial sum.  In this (admittedly quite different) model, it is known with high probability that the number of sign changes of $S(u)$ up to $x$ grows like a multiple of $\sqrt{x}$ as $x \to \infty$; see Chapter III, Section 6, Theorem 2 \cite{feller}.

\section{Proof of Corollary \ref{cor:1}}

\begin{proof}

    In Theorem \ref{thm:mainresult}, fix any $c > 2$ and $A_1 > 1$ and let $A_0 = \frac{1}{7}$. Then, for almost all $f$ there exists a $k_0 := k_0(f, c, A_1) \geq 1$ such that $M_f(u)$ has a sign change in the interval $[y_k, X_k]$ for each $k \geq k_0$, where
    \begin{align*}
        y_k & = \exp \left( \exp \left( \frac{1}{7} \exp ( k^c ) - A_1 k^c \right) \right) \\
        X_k & = \exp \left( \exp \left( \frac{1}{7} \exp ( k^c ) \right) \right).
    \end{align*}
    
    We will show that the intervals $[y_k, X_k]$ are disjoint once $k$ is large enough (depending on $A_1$). 
    Since both $y_k$ and $X_k$ are strictly increasing, it suffices to show that $X_k < y_{k+1}$ for $k$ large enough.
    This in turn follows if we show that
    \begin{equation}\label{eq:inequalitydumb}
    \frac{1}{7} \exp\left( (k+1)^c \right) - A_1 (k+1)^c > \frac{1}{7} \exp(k^c).
    \end{equation}
    To this end, let $\phi(x):= \frac{1}{7}\exp(x^c)$, so that the inequality \eqref{eq:inequalitydumb} reads
    $$\phi(k + 1) - \phi(k)> A_1 (k+1)^c.$$ 
    By the mean-value theorem, for any $k \geq 1$,
    \[
    \phi(k + 1) - \phi(k) \geq \min_{x \in [k, k+1]} \phi'(x)= \min_{x \in [k, k+1]} \frac{1}{7} c x^{c-1} \exp(x^c) 
     = \frac{1}{7} c k^{c-1} \exp(k^c).
    \]
    In the last step we used that $\frac{1}{7} cx^{c-1}\exp(x^c)$ is increasing in $x$, which holds given our choice $c > 2$. 
    Since for $k$ large enough (depending on $c$ and $A_1$), 
    \[
    \frac{1}{7}ck^{c-1} \exp(k^c) > A_1 (k+1)^c.
    \]
   the inequality \eqref{eq:inequalitydumb} follows. So, by taking $k_0$ even larger if necessary,
   the intervals $[y_k, X_k]$ are disjoint once $k\ge k_0$.

    Let $f$ be any multiplicative function in the full measure set furnished by Theorem \ref{thm:mainresult}. 
    Let $V_f(x)$ denote the number of sign changes less than $x$ for the summatory function $M_f(u)$. Since the intervals $[y_k, X_k]$ are disjoint once $k\ge k_0$, we can get a lower bound on $V_f(x)$ by counting those intervals with $X_k\le x$. 
    For $x > 10^4$, say, $X_k\le x$ is ensured if
    \[
    k \leq \left(\log \left(7 \log \log x \right) \right)^{1/c}.
    \]
    Thus, for each $f$ in the full measure set supplied by Theorem \ref{thm:mainresult}, we have
    \[
    V_f(x) \geq \left(\log \left(7 \log \log x \right) \right)^{1/c} - C,
    \]
    where $C= k_0$ is the point where the $[y_k, X_k]$ are guaranteed to all be disjoint.
\end{proof}

\section{Required Propositions for Proof of Theorem \ref{thm:mainresult}}

\begin{prop}[Theorem 1.1, \cite{tenenbaum}]\label{prop:bigO_result}
    Let $\epsilon > 0$. Then, for almost all random multiplicative functions $f$, there exists $x_0 := x_0(f) \geq 1$ such that
    \[
    |M(x)| \leq 4 \sqrt{x} (\log \log x)^{2 + \epsilon}.
    \]
    for all $x \geq x_0.$
\end{prop}

\begin{remark}
    A recent and exciting result of Caich \cite{caich2023} improves the bound to
    \[
    |M(x)| \ll \sqrt{x} (\log \log x)^{1/4 + \epsilon},
    \]
    almost surely. This matches an almost sure lower bound of Harper in \cite{harper2023almost}. Therefore, this result concludes the 80-year story of proving a variant of the law of iterated logarithms for $M(x)$ which started in 1944 with the work of Wintner \cite{wintner1944}.
\end{remark}

\begin{prop}\label{prop:sigmagoing1/2}
    The random Dirichlet series $F(s)$ satisfies
    \[
    F(\sigma) = (2\sigma - 1)^{1/2 + o(1)}
    \]
    as $\sigma \to 1/2^+$, almost surely.
\end{prop}

\begin{proof}
    Since $F(\sigma)$ has an Euler product for $\sigma > 1/2$ almost surely \cite{wintner1944}, we can take the logarithm and expand to get
    \begin{align*}
        \log F(\sigma) & = \log \prod_p \left(1 + \frac{f(p)}{p^{\sigma}} \right) \\
               & = \sum_p \frac{f(p)}{p^{\sigma}} - \frac{1}{2} \sum_{p} \frac{1}{p^{2\sigma}} + O(1) \\
               & = \sum_p \frac{f(p)}{p^{\sigma}} - \frac{\log \zeta(2\sigma)}{2} + O(1),
    \end{align*}
    almost surely. The logarithm branch refers to one where $\lim_{\sigma\to +\infty} \log F(\sigma)=0$. 
    Fix any positive $\epsilon_1 <1/2$ and any $c_1>\sqrt{2}$. By Theorem \ref{thm:2}, there is  almost surely 
    a constant $\sigma_0:=\sigma_0(f)>1/2$ such that 
    \[
    \sum_{p} \frac{f(p)}{p^{\sigma}} \le c_1 \left( \log \left( \frac{1}{2\sigma - 1} \right) \right)^{1/2 + \epsilon_1}
    \]
    for all $\sigma\in (1/2,\sigma_0)$. Furthermore, by Proposition \ref{prop:zetaasym},
    \[
    \log \zeta(2\sigma) \sim \log \left( \frac{1}{2\sigma - 1} \right)
    \]
    as $2\sigma \to 1^+$. Therefore, almost surely, as $\sigma \to 1/2^+$,
    \[
    \log F(\sigma) \le c_1 \left( \log \left( \frac{1}{2\sigma - 1} \right) \right)^{1/2 + \epsilon_1} - \frac{1+o(1)}{2} \log \left( \frac{1}{2\sigma - 1} \right).
    \]
    Since $\epsilon_1<1/2$, almost surely $\log F(\sigma)= (1/2 + o(1)) \log (2\sigma - 1)$ as $\sigma \to 1/2^+$. So, exponentiating both sides in the last asymptotic relation, the proposition follows.

\end{proof}

\begin{remark}
    The proof of Proposition~\ref{prop:sigmagoing1/2} is included as a matter of convenience in \S{\ref{sec:thm2}}. The proof is essentially identical to that of Lemma 2.6 in \cite{sign_changes} except we utilized our Theorem \ref{thm:2} which is a stronger form of Lemma 2.1 in \cite{sign_changes}.
\end{remark}

The following proposition is proved in \S{\ref{sec:harper}}.

\begin{prop}\label{prop:lowerboundresult}
    Let $c$, $C_0$, $C_1$, and $C_2$ be any constants such that 
    $c > 2$, $0<C_0 < 1/2,$ $C_1 > 1$, and $-1.765<C_2< -1.419$. Let
    \[
    \sigma_k := \frac{1}{2} + \frac{1}{\exp \left( \exp(k^c) \right)}.
    \]
    For almost all random multiplicative functions $f$ there exists $k_0 := k_0(f, c, C_0, C_1, C_2) \geq 1$ such that
    \[
    \exp\left( C_0 \log \left( \frac{1}{\sigma_k - 1/2} \right) - C_1 \log \log \left( \frac{1}{\sigma_k - 1/2} \right) + C_2  \right) \leq \sup_{1 \leq t \leq 2 \log^2 \left( \frac{1}{\sigma_k - \frac{1}{2}} \right)} |F(\sigma_k + it)|,
    \]
  for each $k \geq k_0.$
\end{prop}

\section{Proof of Theorem \ref{thm:mainresult}}

\begin{proof}
    Take any $c > 2$, $C_0 \in (0, 1/2)$, $C_1 > 1$ and $C_2 \in (-1.765, -1.418)$. For $k\ge 1$, let 
    \[
    \sigma_k := \frac{1}{2} + \frac{1}{\exp \left(\exp(k^c) \right)},\qquad
        X_k := \exp \left( \frac{1}{(\sigma_k - 1/2)^2} \right).
    \]
    Let $\mathcal{F}$ denote the set of random multiplicative functions $f$ satisfying the almost sure assertions in Propositions \ref{prop:bigO_result}, \ref{prop:sigmagoing1/2} and \ref{prop:lowerboundresult}. So, $\mathcal{F}$ has full measure. 
    
    As $k \to \infty$, $\sigma_k \to 1/2^+$ and $X_k \to \infty$. Therefore, by propositions \ref{prop:bigO_result}, \ref{prop:sigmagoing1/2}, and \ref{prop:lowerboundresult},  given $f \in \mathcal{F}$ there exists a constant 
    \begin{equation*}\label{k0 def}
        k_0 := k_0(f, c, C_0, C_1, C_2) \geq 1 
    \end{equation*}
    such that, generously,
    \begin{equation}\label{eq:upperboundinequality}
        |M_f(x)|  \leq 4 \sqrt{x} \log^2 x\qquad \text{for all }x \geq X_{k_0},
    \end{equation}
    (this is consequence of Proposition \ref{prop:bigO_result} with $k_0$ depending on $f$ and $c$), and
    \begin{equation}\label{eq:Fsigmabehavior}
        |F(\sigma_k)/\sigma_k| \leq 1/2\qquad \text{for every } k \geq k_0, 
    \end{equation}
    (this is a mild consequence of Proposition \ref{prop:sigmagoing1/2} with $k_0$ depending on $f$ and $c$), and
    \begin{align}\label{eq:suplowerbound}
        \exp &\left( C_0 \log \left(  \frac{1}{\sigma_k - 1/2} \right) - C_1 \log \log \left( \frac{1}{\sigma_k - 1/2} \right) + C_2  \right) \leq  \\ \nonumber 
       & \sup_{1 \leq t \leq 2 \log^2 \left( \frac{1}{\sigma_k - \frac{1}{2}} \right)} |F(\sigma_k + it)| \qquad 
       \text{for every }k \geq k_0,
    \end{align}
    (this is consequence of Proposition \ref{prop:lowerboundresult} 
    with $k_0$ depending on $f$, $c$, $C_0$, $C_1$, and $C_2$).

    Let $\{y_k\}$ be a strictly increasing sequence which will be chosen later subject to $y_k < X_k$ for all $k \geq 1$ and $y_k \to \infty$ as $k \to \infty$. Consider the two sequences of events
    \[
    \mathcal{A}_k^+ := \mathcal{A}(y_k,X_k)^+ = \left[ M_f(u) \geq 0 \text{  for } u \in [y_k, X_k] \right],
    \]
    and
    \[
    \mathcal{A}_k^- := \mathcal{A}(y_k,X_k)^- = \left[ M_f(u) \leq 0 \text{  for } u \in [y_k, X_k] \right],
    \]
    for $k \geq 1$. Furthermore, letting $\overline{A_k^+}$ denote the complement of $A_k^+$ (and similarly for $A_k^-$), consider the event
    \[
    \mathcal{A} = \liminf_{k \to \infty} \left( \overline{\mathcal{A}_k^+} \cap \overline{\mathcal{A}_k^-} \right).
    \]
    
    Note that $f$ belongs to $\mathcal{A}$ if and only if there exists $k' \geq 1$ such that $f \in \overline{\mathcal{A}_k^+} \cap \overline{\mathcal{A}_k^-}$ for all $k \geq k'$. Or, what is the same, if and only if $f$ is in neither $A_k^+$ nor $A_k^-$ for all $k\ge k'$. In particular, if $f\in \mathcal{A}$, then $M_f(u)$ has a sign change in the intervals $[y_k, X_k]$ for all $k\ge k'$. Informally, one may think of $\mathcal{A}$ as the event that
   $M_f(u)$ has a sign change in all but finitely many of the intervals $[y_k, X_k]$.

    Our plan is to choose the sequence $\{y_k\}$ so that if $f \in \mathcal{F}$ then $f \in \mathcal{A}$. Equivalently, so that $\mathcal{F}\subseteq \mathcal{A}$. Since $\mathbb{P}(\mathcal{F}) = 1$, we would thus obtain $\mathbb{P}(\mathcal{A}) = 1$, which would complete the proof of Theorem~\ref{thm:mainresult}.
   
    To this end, we argue by contradiction. Let $f \in \mathcal{F}$ and suppose $f \not\in \mathcal{A}$. 
    In other words, suppose $f \in \mathcal{A}_k^+ \cup \mathcal{A}_k^-$ for infinitely many $k$.  
    Let $\mathcal{I}(f)$ denote the collection of $k$ such that $f \in \mathcal{A}_k^+ \cup \mathcal{A}_k^-$.
    Moreover, recall that $k_0$ is chosen to ensure that the conditions \eqref{eq:upperboundinequality}, \eqref{eq:Fsigmabehavior}, and \eqref{eq:suplowerbound} hold for $f$. 
    
    Starting with the supremum inequality in (\ref{eq:suplowerbound}), we find an upper bound for
    \[
    \sup_{1 \leq t \leq 2 \log^2 \left( \frac{1}{\sigma_k - \frac{1}{2}} \right)} |F(\sigma_k + it)|
    \]
    for each $k \in \mathcal{I}(f)$ sufficiently large. 
    By an application of partial summation (Theorem 1.3, \cite{montgomery}), we have
    \begin{equation}\label{partial summation}
    F(\sigma + it) = (\sigma + it) \int_1^{\infty} \frac{M_f(u)}{u^{1 + \sigma + it}} \ du
     \end{equation}
    for any $\sigma > 1/2$ and $t \in \mathbb{R}$, almost surely. So,
\begin{align}
     \sup_{1 \leq t \leq 2 \log^2 \left( \frac{1}{\sigma_k - \frac{1}{2}} \right)} |F(\sigma_k + it)|  & \leq  \sup_{1 \leq t \leq 2 \log^2 \left( \frac{1}{\sigma_k - \frac{1}{2}} \right)} |\sigma_k + it| \int_1^{\infty} \frac{|M_f(u)|}{u^{1 + \sigma_k}} \ du \notag \\
        & \leq  C_4 \log^2 \left( \frac{1}{\sigma_k - 1/2} \right) \int_1^{\infty} \frac{|M_f(u)|}{u^{1 + \sigma_k}} \ du \label{eq:step2inequality}
\end{align}
for any constant $C_4 > 2$ and any $k$ sufficiently large\footnote{It may be necessary to make $k$ larger in order to ensure the inequality
\[
\left| \sigma_k + 2i \log^2 \left( \frac{1}{\sigma_k - 1/2} \right) \right| \leq C_4 \log^2 \left( \frac{1}{\sigma_k - \frac{1}{2}} \right) 
\]
holds for the chosen $C_4$.}. 

We now focus on the integral
\begin{equation}\label{eq:main_integral}
\int_1^{\infty} \frac{|M_f(u)|}{u^{1 + \sigma_k}} \ du.
\end{equation}
Since $f$ satisfies the upper bound in (\ref{eq:upperboundinequality}) for all $u \geq X_{k_0}$,
the tail of the integral is bounded by
\begin{align}
        \int_{X_k}^{\infty} \frac{|M_f(u)|}{u^{1 + \sigma_k}} \ du  & \leq 4 \int_{X_k}^{\infty} \frac{\sqrt{u} \log^2 u}{u^{1 + \sigma_k}} \ du \notag \\
                 & = \frac{4 \log^2 X_k}{(\sigma_k - 1/2)X_k^{\sigma_k - 1/2}} + \frac{8 \log X_k}{(\sigma_k - 1/2)^2 X_k^{\sigma_k - 1/2}} + \frac{8}{(\sigma_k - 1/2)^3 X_k^{\sigma_k - 1/2}} \notag \\
                 & = \frac{4 + 8(\sigma_k - 1/2) + 8(\sigma_k - 1/2)^2 }{(\sigma_k  -1/2)^5 \exp\Big( (\sigma_k - 1/2)^{-1} \Big)} = o(1) \notag
\end{align}
as $k \to \infty$ and since $\sigma_k \to 1/2^+.$ Therefore, taking $k_0$ even larger if necessary, the tail satisfies
\begin{equation}\label{eq:tailintegralestimate}
    \int_{X_k}^{\infty} \frac{|M_f(u)|}{u^{1 + \sigma_k}} \ du < 1/2
\end{equation}
for each $k \geq k_0.$
Next, we bound the initial part of integral (\ref{eq:main_integral}), namely
\[
\int_1^{X_k} \frac{|M_f(u)|}{u^{1 + \sigma_k}} \ du,
\]
for each $k \in \mathcal{I}(f)$ sufficiently large. We note
by \eqref{eq:Fsigmabehavior} that for all $k\ge k_0$, we have $|F(\sigma_k)/\sigma_k| \leq 1/2$. Therefore, by the partial summation formula \eqref{partial summation},
\[
\left| \int_1^{\infty} \frac{M_f(u)}{u^{1 + \sigma_k}} \ du \right| \le 1/2.
\]
Combining this with the reverse triangle inequality,
\[
\left| \int_1^{\infty} \frac{M_f(u)}{u^{1 + \sigma_k}} \ du \right|\ge 
\left| \int_1^{X_k} \frac{M_f(u)}{u^{1 + \sigma_k}} \ du \right|-
 \int_{X_k}^{\infty} \frac{|M_f(u)|}{u^{1 + \sigma_k}} \ du,
\]
as well as the tail integral estimate in (\ref{eq:tailintegralestimate}),
we obtain for each $k \geq k_0$,
\begin{equation}\label{Xk integral}
    \left| \int_1^{X_k} \frac{M_f(u)}{u^{1 + \sigma_k}} \ du \right| < 1.
\end{equation}

We now split the integral \eqref{Xk integral} at $y_k$ and use the reverse triangle inequality once again to get
\begin{equation}\label{eq:usefulbit}
\left| \int_{y_k}^{X_k} \frac{M_f(u)}{u^{1 + \sigma_k}} \ du \right| < \left| \int_1^{y_k} \frac{M_f(u)}{u^{1 + \sigma_k}} \ du \right| + 1.
\end{equation}
We bound the integral from 1 to $y_k$ using the upper bound in \eqref{eq:upperboundinequality}.  
This yields
\begin{align}
    \left| \int_1^{y_k} \frac{M_f(u)}{u^{1 + \sigma_k}} \ du \right| & \leq \int_{1}^{y_k} \frac{|M_f(u)|}{u^{1 + \sigma_k}} \ du \notag \\
                                & \le \int_1^{X_{k_0}} \frac{|M_f(u)|}{u^{1 + \sigma_k}} \ du + \mathds{1}_{y_k>X_{k_0}}\int_{X_{k_0}}^{y_k} \frac{|M_f(u)|}{u^{1 + \sigma_k}} \ du \notag \\
                                & \leq \int_1^{X_{k_0}} \frac{u}{u^{1 + \sigma_k}} \ du + \mathds{1}_{y_k>X_{k_0}} 4 \int_{X_{k_0}}^{y_k} \frac{\sqrt{u} \log^2 u}{u^{1 + \sigma_k}} \ du \notag \\
                                & \leq (X_{k_0} - 1) + \mathds{1}_{y_k>X_{k_0}}\frac{4}{3} \left( \log^3 y_k - \log^3 X_{k_0} \right) \notag\\
                                & \leq C_5 \log^3 y_k, \label{eq:jeb1}
\end{align}
for any constant $C_5 > 4/3$ and for each $k \geq k_0$, again taking $k_0$ larger if needed.

Now, take any $k \in \mathcal{I}(f)$. So, $f \in \mathcal{A}_k^+$ or $f \in \mathcal{A}_k^-$. In either case, $M_f(u)$ is of constant sign throughout $[y_k,X_k]$. Therefore,
\begin{equation}\label{abs val eq}
\left| \int_{y_k}^{X_k} \frac{M_f(u)}{u^{1 + \sigma_k}} \ du \right| = \int_{y_k}^{X_k} \frac{|M_f(u)|}{u^{1 + \sigma_k}} \ du.
\end{equation}
Combining \eqref{abs val eq}, \eqref{eq:jeb1}, and \eqref{eq:usefulbit}, we see that for $k \in \mathcal{I}(f)$ satisfying $k \geq k_0$ we may bound the initial part of the main integral (\ref{eq:main_integral}) by
\begin{align*}
    \int_1^{X_k} \frac{|M_f(u)|}{u^{1 + \sigma_k}} \ du & = \int_1^{y_k} \frac{|M_f(u)|}{u^{1 + \sigma_k}} \ du +  \int_{y_k}^{X_k} \frac{|M_f(u)|}{u^{1+ \sigma_k}} \ du \\
                    & = \int_1^{y_k} \frac{|M_f(u)|}{u^{1 + \sigma_k}} \ du +  \left| \int_{y_k}^{X_k} \frac{M_f(u)}{u^{1+ \sigma_k}} \ du \right| \\
                    & <  \int_1^{y_k} \frac{|M_f(u)|}{u^{1 + \sigma_k}} \ du + \left| \int_{1}^{y_k} \frac{M_f(u)}{u^{1+ \sigma_k}} \ du \right| + 1 \\
                    & \leq 2 C_5 \log^3 y_k + 1.
\end{align*}
The first inequality above comes from applying bound \eqref{eq:usefulbit} to the second integral and the second inequality comes from applying bound \eqref{eq:jeb1} to both integrals. 

Overall, we have that $f\in \mathcal{F}$ satisfies
\begin{equation}\label{eq:initialpart}
\int_1^{X_k} \frac{|M_f(u)|}{u^{1 + \sigma_k}} \ du \leq C_6 \log^3 y_k,
\end{equation}
where $C_6 > 2C_5 > 8/3$ and for each $k \in \mathcal{I}(f)$ satisfying $k \geq k_0$, taking $k_0$ larger if necessary.
Combining the estimates for the initial part (\ref{eq:initialpart}) and the tail part (\ref{eq:tailintegralestimate}) of integral (\ref{eq:main_integral}), we obtain
\[
    \int_1^{\infty} \frac{|M_f(u)|}{u^{1 + \sigma_k}} \ du \leq C_6 \log^3 y_k
\]
for each $k \in \mathcal{I}(f)$ with $k \geq k_0$. Returning to the inequality in (\ref{eq:step2inequality}), we thus have
\begin{equation}\label{eq:thirdinequality}
    \sup_{1 \leq t \leq 2 \log^2 \left( \frac{1}{\sigma_k - \frac{1}{2}} \right)} |F(\sigma_k + it)|  \leq C_7 \log^2 \left( \frac{1}{\sigma_k - 1/2} \right) \log^3 y_k,
\end{equation}
where $C_7 = C_4C_6$. The inequality (\ref{eq:thirdinequality}) together with the supremum inequality in \eqref{eq:suplowerbound} therefore give that
\[
\exp\left( C_0 \log \left( \frac{1}{\sigma_k - 1/2} \right) - C_1 \log \log \left( \frac{1}{\sigma_k - 1/2} \right) + C_2  \right) \leq C_7 \log^2 \left( \frac{1}{\sigma_k - 1/2} \right) \log^3 y_k
\]
holds for each $k \in \mathcal{I}(f)$ with $k \geq k_0$.
After some rearranging, we hence deduce under our hypothesis to the contrary that
\begin{equation*}\label{A2A1A0}
\exp \left( A_2 \left( \frac{1}{\sigma_k - 1/2} \right)^{A_0} \log^{-A_1} \left( \frac{1}{\sigma_k - 1/2} \right) \right) \leq y_k,    
\end{equation*}
where $A_0 := C_0/3$, $A_1 := (C_1 + 2)/3$, $A_2:=e^{C_2/3}(C_7)^{-1/3}$, and $k \in \mathcal{I}(f)$ with $k \geq k_0.$ Therefore, we arrive at a contradiction if $k\ge k_0$, $k\in \mathcal{I}(f)$, and $y_k$ satisfies
\[
    y_k < \exp \left( A_2\left( \frac{1}{\sigma_k - 1/2} \right)^{A_0} \log^{-A_1} \left( \frac{1}{\sigma_k - 1/2} \right) \right).
\]
So, it must be that $f \in \overline{\mathcal{A}_k^+} \cap \overline{\mathcal{A}_k^-}$ for each $k$ such that $k \geq k_0$ and $k \in \mathcal{I}(f)$. In other words, it must be that $f\in \mathcal{A}$, as we wished to prove.

Lastly, given the conditions on the constants $C_1,\ldots,C_7$, we have the constraints $A_0 \in (0, 1/6)$, $A_1 > 1$ and $A_2 > 0$. So, setting $A_2=2$ say, and letting $k_0$ be resulting threshold on $k$,  we  
still obtain a contradiction if $f\in \mathcal{F}$, $f\not\in \mathcal{A}$, $k\ge k_0$, $k\in \mathcal{I}(f)$, 
and on taking $y_k$ to be of the simpler form
\[
 y_k = \exp \left( \left(\frac{1}{\sigma_k - 1/2} \right)^{A_0} \log^{-A_1} \left( \frac{1}{\sigma_k - 1/2} \right) \right).
\]
\end{proof}

\section{Proof of Theorem \ref{thm:2}}\label{sec:thm2}

Let
\begin{equation}\label{P prime sum}
    P(\sigma) := \sum_{p} \frac{f(p)}{p^{\sigma}} \qquad \text{and} \qquad
    \overline{P}(\sigma) := \frac{P(\sigma)}{\sqrt{\mathbb{E}[P(\sigma)^2]}}.
\end{equation}
By the Kolmogorov Three-Series Theorem, which is stated as 
Proposition \ref{prop:twoseries} in Appendix B, $P(\sigma)$ is convergent on the half-plane $\sigma > 1/2$ for almost all $f$.
Moreover, for such $f$,  $P(\sigma)$ defines an analytic function in that half-plane,  
as follows from basic properties of Dirichlet series; see \cite[\S{11}]{apostol_1976}.

Since $\mathbb{E}[f(p)] = 0$, we have for any $\sigma > 1/2$ that $\mathbb{E}[P(\sigma)] = 0$ and
\begin{equation}\label{Psigma var}
\mathbb{E} [ P(\sigma)^2] = \sum_p \frac{1}{p^{2\sigma}}.    
\end{equation}
Moreover, by Proposition \ref{prop:zetaasym},
\begin{equation}\label{eq:behavior_of_P}
\sum_{p} \frac{1}{p^{2\sigma}} \sim \log \left( \frac{1}{2\sigma - 1} \right)
\end{equation}
as $\sigma \to 1/2^+$. So, put together, Theorem \ref{thm:2} is equivalent to the statement that for any $\epsilon_1 > 0$, almost surely,
\[
\overline{P}(\sigma) \ll \mathbb{E}[P(\sigma)^2]^{\epsilon_1}
\]
as $\sigma \to 1/2^+$. 

To prove this statement, we follow the suggestion in \cite[Remark 2.1]{sign_changes}. So, in effect, we  
complete steps 2--4 in \cite{LIL}. 
It is worth mentioning that the context of \cite{LIL} is to obtain 
a law of iterated logarithms for random Dirichlet series $\sum_{n \geq 1} X_n/n^{s}$, where the $X_n$ are i.i.d. random variables uniformly distributed on $\{\pm 1\}$. 

To this end, we will use the following two lemmas, corresponding to steps 2 and 3 in \cite{LIL}, which we prove later in this section.

\begin{lemma}[Step 2, \cite{LIL}]
\label{lem:step2}
    Let $\epsilon$ be a positive constant less than $2$, and set $\delta := \epsilon/2$. For $\ell\ge 1$, let
    \begin{equation}\label{sigma ell def}
        \sigma_{\ell} := \frac{1}{2} + \frac{1}{2\exp(\ell^{1 - \delta})}.
    \end{equation}
    For any $\gamma > 0$, we have
    \[
    \limsup_{\ell \to \infty} \frac{\overline{P}(\sigma_{\ell})}{\sqrt{2\, \mathbb{E}[P(\sigma_{\ell})^2]^{\epsilon}}} \leq \sqrt{1 + \gamma},
    \]
    almost surely.
\end{lemma}

\begin{lemma}[Step 3, \cite{LIL}]\label{lemma3.4}\label{lem:step3}
    Let $\epsilon$ be a positive constant less than $2$, and set $\delta := \epsilon/2$. Let $\sigma_{\ell}$ be defined as in \eqref{sigma ell def}. There exists a universal constant $C > 0$ such that for almost all random multiplicative functions $f$ there is a contant $\ell_0 := \ell_0(f) \geq 1$ such that
    \[
    \max_{\sigma \in [\sigma_{\ell}, \sigma_{\ell-1}]} |P(\sigma) - P(\sigma_{\ell})| \leq C,
    \]
    for all $\ell \geq \ell_0$.
\end{lemma}

\begin{proof}[Proof of Theorem \ref{thm:2}]
In preparation for applying lemmas \ref{lem:step2} and \ref{lem:step3}, pick any $\gamma>0$ and any $0<\epsilon<2$. 
Set $\delta = \epsilon/2$ and 
\begin{equation}\label{eq:sigmak}
\sigma_{\ell} = \frac{1}{2} + \frac{1}{2\exp(\ell^{1 - \delta})}.
\end{equation}
By Lemma \ref{lemma3.4}, there exists a universal constant $C > 0$ such that for almost all random multiplicative functions $f$ there is a constant $\ell_0 \geq 1$, depending on $f$, so that
\[
\max_{\sigma \in [\sigma_{\ell}, \sigma_{\ell-1}]} |P(\sigma) - P(\sigma_{\ell})| \leq C
\]
for all $\ell \geq \ell_0$ as $\sigma_{\ell} < 1$. Let $f$ be any random multiplicative function satisfying, both, this property 
as well as the bound in Lemma \ref{lem:step2} (with the same $\gamma$ and $\epsilon$ we picked at the beginning).
Writing
\begin{equation}\label{eq:2parts}
    \frac{\overline{P}(\sigma)}{\sqrt{2\,\mathbb{E}[P(\sigma)^2]^{\epsilon}}} = \frac{P(\sigma_{\ell})}{\sqrt{2\, \mathbb{E}[P(\sigma)^2] \mathbb{E}[P(\sigma)^2]^{\epsilon}}} + \frac{P(\sigma) - P(\sigma_{\ell})}{\sqrt{2\, \mathbb{E}[P(\sigma)^2] \mathbb{E}[P(\sigma)^2]^{\epsilon}}},
\end{equation}
our plan is to bound each of the ratios on the right-side of \eqref{eq:2parts}.

Since $f$ satisfies the bound in Lemma \ref{lem:step2}, we have, 
on increasing $\ell_0$ if necessary, that for any constant $\gamma_1>\gamma$ 
the first ratio in \eqref{eq:2parts} is
\begin{align}\label{eq:2parts 1}
    &\leq \sqrt{1 + \gamma_1} \frac{\sqrt{\mathbb{E}[P(\sigma_{\ell})^2] \mathbb{E}[P(\sigma_{\ell})^2]^{\epsilon}}}{\sqrt{\mathbb{E}[P(\sigma)^2] \mathbb{E}[P(\sigma)^2]^{\epsilon}}}
\end{align}
for $\ell\ge \ell_0$. Now, suppose $\sigma \leq \sigma_{\ell-1}$.
By the formula \eqref{Psigma var}, $\mathbb{E}[P(\sigma)^2] \geq \mathbb{E}[P(\sigma_{\ell-1})^2]$. So, the expression in \eqref{eq:2parts 1} is
\begin{align*}
     &\leq \sqrt{1 + \gamma_1} \left( \frac{\mathbb{E}[P(\sigma_{\ell})^2]}{\mathbb{E}[P(\sigma_{\ell-1})^2]} \right)^{(1 + \epsilon)/2}.
\end{align*}
Furthermore, by the asymptotic for the variance of $P(\sigma_k)$ given in (\ref{eq:behavior_of_P}), 
\begin{align*}
\mathbb{E}[P(\sigma_{\ell})^2] & = \sum_{p} \frac{1}{p^{2\sigma_{\ell}}} \sim \log \left( \frac{1}{2\sigma_{\ell} - 1} \right) = \ell^{1- \delta}
\end{align*}
as $\ell \to \infty$. Therefore,
\[
\lim_{\ell \to \infty} \left( \frac{\mathbb{E}[P(\sigma_{\ell})^2]}{\mathbb{E}[P(\sigma_{\ell-1})^2]} \right)^{(1 + \epsilon)/2} = 1.
\]
As a result, and on increasing $\ell_0$ again if needed, we have for any constant $\gamma_2>\gamma_1$ that 
the first ratio in \eqref{eq:2parts} satisfies
\begin{equation} \label{first ratio}
    \limsup_{\sigma \to 1/2^+} \frac{P(\sigma_{\ell})}{\sqrt{2\, \mathbb{E}[P(\sigma)^2] \mathbb{E}[P(\sigma)^2]^{\epsilon}}} \leq \sqrt{1 + \gamma_2},
\end{equation}
for $\ell\ge \ell_0$. 

The second ratio in \eqref{eq:2parts} can be dealt with in a similar fashion. Since $f$ satisfies the bound in
Lemma \ref{lemma3.4}, we have
\begin{align*}
\max_{\sigma \in [\sigma_{\ell}, \sigma_{\ell-1}]} \left| \frac{P(\sigma) - P(\sigma_{\ell})}{\sqrt{2\, \mathbb{E}[P(\sigma)^2] \mathbb{E}[P(\sigma)^2]^{\epsilon}}} \right| 
                    & \leq \frac{C}{{\sqrt{\mathbb{E}[P(\sigma_{\ell-1})^2] \mathbb{E}[P(\sigma_{\ell-1})^2]^{\epsilon}}}}.
\end{align*}
for $\ell \ge \ell_0$. Moreover, since 
$\mathbb{E}[P(\sigma_{\ell})^2] \to \infty$ as $\ell \to \infty$, the last quantity tends to $0$ with $\ell$. 
Thus, for any $\gamma_3>0$, and on increasing $\ell_0$ yet again if needed, we obtain
\begin{equation}\label{second ratio}
     \limsup_{\sigma \to 1/2^+} \left|\frac{P(\sigma) - P(\sigma_{\ell})}{\sqrt{2\, \mathbb{E}[P(\sigma)^2] \mathbb{E}[P(\sigma)^2]^{\epsilon}}}\right| < \gamma_3,
\end{equation}
for $\ell\ge \ell_0$. 

Put together, using the information \eqref{first ratio} and \eqref{second ratio} back in
\eqref{eq:2parts}, we have that 
\begin{align*}
\limsup_{\sigma \to 1/2^+} \frac{\overline{P}(\sigma)}{\sqrt{2\,\mathbb{E}[P(\sigma)^2]^{\epsilon}}} \leq \sqrt{1 + \gamma_2}+\gamma_3,
\end{align*}
for $\ell\ge \ell_0$. 
Since $\gamma_2$ and $\gamma_3$ can be made arbitrarily small, we deduce for any constant $c_1 > \sqrt{2}$ that
\begin{equation}\label{Pbar sigma}
    \overline{P}(\sigma) \leq c_1\, \mathbb{E}[P(\sigma)^2]^{\epsilon/2} \sim c_1 \left( \log \left( \frac{1}{2\sigma - 1} \right) \right)^{\epsilon/2},
\end{equation}
where the inequality holds for all $\sigma>1/2$ sufficiently close to $1/2$, and the 
asymptotic is taken  as $\sigma\to 1/2^+$. Letting $\epsilon_1 := \epsilon/2$ completes the proof.  Lastly, although the bound \eqref{Pbar sigma} is derived assuming that the constant $\epsilon$ satisfies $0<\epsilon<2$, it is clear that \eqref{Pbar sigma} still holds even if $\epsilon \ge 2$.
    
\end{proof}

\subsection{Proof of Lemma \ref{lem:step2}}

\begin{proof}

    Let $\gamma$ be any positive constant. By formula \eqref{Psigma var} and Hoeffding's inequality, which is stated as Proposition \ref{thm:hoeffding} in Appendix B, we have
    \begin{equation}\label{probability comparison}
         \mathbb{P} \left( \overline{P}(\sigma_{\ell}) \geq \sqrt{2(1 + \gamma) \mathbb{E}[P(\sigma_{\ell})^2]^{\epsilon}} \right) \leq \exp \left( - (1 + \gamma)\mathbb{E}[P(\sigma_{\ell})^2]^{\epsilon} \right).
    \end{equation}
    Furthermore, by the asymptotic for $\mathbb{E}[P(\sigma)^2]$ in \eqref{eq:behavior_of_P}, and since $\sigma_{\ell} \to 1/2^+$ as $\ell \to \infty$,
    \[
    \mathbb{E}[P(\sigma_{\ell})^2] = \sum_{p} \frac{1}{p^{2\sigma_{\ell}}} \sim \log \left( \frac{1}{2 \sigma_{\ell} - 1} \right) = \ell^{1 - \delta}
    \]
    as $\ell \to \infty$. Thus
    \begin{equation}\label{limit comparison}
         \exp \left( - (1 + \gamma)\mathbb{E}[P(\sigma_{\ell})^2]^{\epsilon} \right)
                \sim  \exp \left( - (1 + \gamma)\ell^{(1 - \delta)\epsilon} \right).
    \end{equation}
    Since $\epsilon < 2$, $1-\delta >0$, and so the series $\sum_{k \geq 1} \exp( -(1 + \epsilon) k^{(1-\delta)\epsilon})$ converges, considering $\epsilon>0$.
    So, by \eqref{probability comparison}, \eqref{limit comparison}, and the limit comparison test, it follows that
    \[
    \sum_{k \geq 1}\mathbb{P}  \left( \overline{P}(\sigma_{\ell})  \geq \sqrt{2(1 + \gamma) \mathbb{E}[P(\sigma_{\ell})^2]^{\epsilon}} \right) < \infty.
    \]
    It therefore follows by the Borel-Cantelli lemma that
    \[
    \mathbb{P} \left( \limsup_{\ell \to \infty} \left[ \frac{\overline{P}(\sigma_{\ell})}{\sqrt{2\, \mathbb{E} {P(\sigma_{\ell})^2]^{\epsilon}}}} \geq \sqrt{1 + \gamma} \right] \right) = 0.
    \]
    Hence,
    \[
    \mathbb{P} \left( \liminf_{\ell \to \infty} \left[ \frac{\overline{P}(\sigma_{\ell})}{\sqrt{2\, \mathbb{E} {P(\sigma_{\ell})^2]^{\epsilon}}}} \leq \sqrt{1 + \gamma} \right] \right) = 1,
    \]
    which in turn implies
    \[
    \limsup_{\ell \to \infty} \frac{\overline{P}(\sigma_{\ell})}{\sqrt{2\, \mathbb{E} {P(\sigma_{\ell})^2]^{\epsilon}}}} \leq \sqrt{1 + \gamma},
    \]
    almost surely, as we wished to show.
    
\end{proof}

\subsection{Proof of Lemma \ref{lem:step3}}

\begin{proof}

We will prove this lemma by working with a certain decomposition of the intervals $[\sigma_{\ell}, \sigma_{\ell-1}]$, 
where $\ell \geq 2$. For $r \geq 1$, let
\[
D_{\ell,r} := \{\tau_{\ell,r}(n) : 0 \leq n \leq 2^r \}
\]
where
\[
\tau_{\ell,r}(n) := \sigma_{\ell} + \frac{n}{2^r}( \sigma_{\ell-1} - \sigma_{\ell}).
\]
So, $D_{\ell,r}$ denotes the collection of equidistant points with spacing $( \sigma_{\ell-1} - \sigma_{\ell})/2^r$ comprising our decomposition 
of the interval $[\sigma_{\ell}, \sigma_{\ell-1}]$.

Let $\lambda_r$ be a parameter which will be chosen later \textit{independent} of $\ell$. 
(The choice is given in \eqref{eq:lambdadef}.)
Also, recall the definition of $P(\sigma)$ in \eqref{P prime sum}, and consider the events
\[
\mathcal{A}_{\ell, r} := \left[ \max_{0 \leq n < 2^r} \left|P(\tau_{\ell,r}(n+1)) - P(\tau_{\ell,r}(n))\right| \geq \lambda_r \right].
\]
In addition, let us introduce the random variable
\[
U_{\ell}(f) := \min \left\{ u \in \mathbb{N} : f \in \bigcap_{r = u}^{\infty} \mathcal{A}_{\ell, r}^c \right\}.
\]
To help interpret $U_{\ell}(f)$, notice that $P(\sigma)$ is first ensured to vary by an amount $< \lambda_r$ 
between any two consecutive points in $D_{\ell,r}$
exactly when the spacing of the equidistant points is $( \sigma_{\ell-1} - \sigma_{\ell})/2^r$ with $r= U_{\ell}(f)$. Since the series 
$P(s)$ is analytic in the half-plane $\sigma>1/2$, $U_{\ell}(f)$ is finite for 
almost all $f$.

We will now construct the $\lambda_r$ so that
\begin{equation*}\label{eq:ukpp}
\mathbb{P} \left( \liminf_{\ell \to \infty} [U_{\ell} \leq 1] \right) = 1.
\end{equation*}
To this end, we note that 
\[
[U_{\ell} \leq 1] = \bigcap_{r = 1}^{\infty} \mathcal{A}_{\ell, r}^c\qquad \text{and so}\qquad
[U_{\ell} > 1] = [U_{\ell} \leq 1]^c = \bigcup_{r = 1}^{\infty} \mathcal{A}_{\ell, r}.
\]
Therefore,
\begin{align}\label{eq:borel-cantelli-goal}
\mathbb{P}([U_{\ell} > 1]) &\leq \sum_{r = 1}^{\infty} \mathbb{P}(\mathcal{A}_{\ell, r}) \notag \\
&\leq \sum_{r = 1}^{\infty} \sum_{n = 0}^{2^r - 1} \mathbb{P} \left( \left[ \left|P(\tau_{\ell,r}(n+1)) - P(\tau_{\ell, r}(n))\right| \geq \lambda_r \right] \right).
\end{align}

Now, let $f$ be such that $P(\sigma)$ is convergent in the half-plane $\sigma>1/2$. Then,
\[
P(\tau_{\ell,r}(n+1)) - P(\tau_{\ell,r}(n)) = \sum_{p} \left(\frac{1}{p^{\tau_{\ell,r}(n+1)}} - \frac{1}{p^{\tau_{\ell,r}(n)}} \right) f(p).
\]
Observing that the event
$\left[ \left|P(\tau_{\ell,r}(n+1)) - P(\tau_{\ell, r}(n))\right|\geq \lambda_r \right]$
is the union of the events 
$\left[ P(\tau_{\ell,r}(n+1)) - P(\tau_{\ell, r}(n))\geq \lambda_r \right]$ and
$\left[ P(\tau_{\ell,r}(n)) - P(\tau_{\ell, r}(n+1))\geq \lambda_r\right]$,
it follows by Hoeffding's inequality (Proposition \ref{thm:hoeffding}) that
\begin{equation*}
    \mathbb{P} \left( \left[ \left|P(\tau_{\ell,r}(n+1)) - P(\tau_{\ell, r}(n))\right|\geq \lambda_r \right] \right) \le
      2 \exp \left( - \frac{\lambda_r^2}{2 \sum_p \left(p^{-\tau_{\ell,r}(n+1)} - p^{-\tau_{\ell, r}(n)} \right)^2} \right).
\end{equation*}
Moreover, the function $g(\sigma)=p^{-\sigma}$ is differentiable everywhere. 
So, by the mean-value theorem applied to $g(\sigma)$ on the interval $[\tau_{\ell, r}(n), \tau_{\ell,r}(n+1)]$, 
there is a real number $\theta_{\ell,r,p}(n)$ such that  
$\tau_{\ell, r}(n)< \theta_{\ell,r,p}(n)< \tau_{\ell,r}(n+1)$ and
\begin{equation}\label{mvt app}
    \frac{1}{p^{\tau_{\ell,r}(n+1)}} - \frac{1}{p^{\tau_{\ell, r}(n)}} = (\tau_{\ell,r}(n+1) - \tau_{\ell, r}(n)) \frac{\log p}{p^{\theta_{\ell, r,p}(n)}}.
\end{equation}
So, combining our Hoeffding's estimate with
formula \eqref{mvt app} and the observations that
$\sigma_{\ell} < \theta_{\ell, r, p}(n)$ and
$\tau_{\ell, r}(n+1) - \tau_{\ell,r}(n) = (\sigma_{\ell-1} - \sigma_{\ell})2^{-r}$, by construction,
we deduce that $\mathbb{P} \left( \left[ \left|P(\tau_{\ell,r}(n+1)) - P(\tau_{\ell, r}(n))\right|\geq \lambda_r \right] \right)$
is 
\begin{align*}
& \leq 2 \exp \left( - \frac{\lambda_r^2}{2 (\tau_{\ell, r}(n+1) - \tau_{\ell, r}(n))^2 \sum_p (\log p)^2 p^{-2\sigma_{\ell}}} \right) \\
    & \leq 2 \exp\left( - \frac{4^r \lambda_r^2 }{2 (\sigma_{\ell-1} - \sigma_\ell)^2 \sum_p (\log p)^2 p^{-2\sigma_{\ell}}} \right),
\end{align*} 
independent of $n$. 
Combining the last bound with the bound on $\mathbb{P}([U_k > 1])$ in \eqref{eq:borel-cantelli-goal}, we get
\begin{align}
    \mathbb{P}([U_k > 1]) 
                & \le \sum_{r = 1}^{\infty} \sum_{n = 0}^{2^r -1} 2 \exp\left( - \frac{4^r\lambda_r^2}{2 (\sigma_{\ell-1} - \sigma_{\ell})^2 \sum_p (\log p)^2 p^{-2\sigma_{\ell}}} \right) \notag \\
                & = 2 \sum_{r = 1}^{\infty} 2^r\cdot  \exp\left( - \frac{4^r\lambda_r^2}{2 (\sigma_{\ell-1} - \sigma_{\ell})^2 \sum_p (\log p)^2 p^{-2\sigma_{\ell}}} \right) \notag \\
                & = 2 \sum_{r = 1}^{\infty} \exp\left( - \frac{4^r\lambda_r^2}{2 (\sigma_{\ell-1} - \sigma_{\ell})^2 \sum_p (\log p)^2 p^{-2\sigma_{\ell}}} + r \log 2 \right). \label{eq:bigterm}
\end{align}

The next step is to bound the denominator of the last expression in \eqref{eq:bigterm}.
First,, it follows by Lemma \ref{lem:claim1} that there is a universal constant $C_1 > 0$ such that the prime sum in said denominator
satisfies
\begin{equation}\label{eq:bound11}
    \sum_p \frac{\log^2 p}{p^{2\sigma_{\ell}}} \leq  \frac{C_1}{(2 \sigma_{\ell} - 1)^{2}}
\end{equation}
for each $\ell$ satisfying $\sigma_{\ell} < 1$. Second, in preparation for bounding the term $(\sigma_{\ell-1} - \sigma_{\ell})^2$, let
\[
\varphi(x) := \frac{1}{2} + \frac{1}{2 \exp(x^{1 - \delta})}.
\]
And notice that
$\varphi(\ell) = \sigma_{\ell}$, which follows from the definition of $\sigma_{\ell}$ in \eqref{sigma ell def}.
We note as well that
$0<1-\delta<1$ since $0<\epsilon<2$. So, by the mean-value theorem applied to $\varphi$ on 
the interval $[\ell-1, \ell]$, there is a real number $\theta_{\ell}$ such that 
$\ell-1<\theta_{\ell} <\ell$ and
\[
\sigma_{\ell-1} - \sigma_{\ell} = \frac{1-\delta}{2} \cdot \frac{1}{\exp(\theta_{\ell}^{1-\delta})}\cdot \frac{1}{\theta_{\ell}^{\delta}}.
\]
We can bound this as follows
\begin{align*}
\frac{1-\delta}{2} \cdot \frac{1}{\exp(\theta_{\ell}^{1-\delta}) }\cdot \frac{1}{\theta_{\ell}^{\delta}} & 
\leq \frac{1 - \delta}{2}\cdot \frac{1}{\exp((\ell-1)^{1-\delta}) }\cdot \frac{1}{(\ell-1)^{\delta}} \\
& = \frac{1 - \delta}{2} \cdot (2\sigma_{\ell-1} - 1)\cdot\frac{1}{(\ell-1)^{\delta}} 
 \leq \frac{2 \sigma_{\ell} - 1}{\ell^{\delta}}
\end{align*}
for $\ell\ge \ell_1$, where $\ell_1$ is absolute. Therefore,
\begin{equation}\label{eq:subtractionbound}
    \sigma_{\ell-1} - \sigma_{\ell} \leq \frac{2\sigma_{\ell} - 1}{\ell^{\delta}}
\end{equation}
for each $\ell \geq \ell_1$.
Combining the upper bounds \eqref{eq:bound11} and \eqref{eq:subtractionbound} 
with the bound for $\mathbb{P}([U_{\ell} > 1])$ in \eqref{eq:bigterm}, and noting that the 
term $(2\sigma_{\ell}-1)^2$ cancels, we  obtain
\[
\mathbb{P}([U_{\ell} > 1]) \leq  2 \sum_{r = 1}^{\infty} \exp\left( - \frac{4^r \ell^{2\delta}}{2C_1} \lambda_r^2 + r \log 2 \right)
\]
for each $\ell\ge \ell_1$. So, letting
\begin{equation}\label{eq:lambdadef}
    \lambda_r^2 = \frac{2C_1 r }{4^r}
\end{equation}
for $r \geq 1$, we get for any $\ell \geq \ell_1$ that
\[
\mathbb{P}([U_{\ell} > 1]) \leq  2 \sum_{r = 1}^{\infty} \exp \left(\left(-\ell^{2\delta} + \log 2\right) r \right) \leq 
16 e^{-\ell^{2\delta}},
\]
where in the last inequality we used that $\exp( -\ell^{2\delta} + \log 2)\le \exp(-1+\log 2) < 3/4$. 

In summary, since $\ell_1$ is absolute, we have
\[
\sum_{\ell = 1}^{\infty} \mathbb{P}([U_{\ell} > 1]) \ll \sum_{\ell = 1}^{\infty} 16 e^{-\ell^{2\delta}} < \infty.
\]
Hence, by the Borel-Cantelli lemma,
\[
\mathbb{P} \left( \limsup_{\ell \to \infty} [U_{\ell} > 1] \right) = 0
\]
or equivalently
\[
\mathbb{P} \left( \liminf_{\ell \to \infty} [U_{\ell} \leq 1] \right) = 1.
\]
Therefore, for almost all $f$ there exists a $\ell_0:=\ell_0(f) \geq 1$ such that
\[
\max_{0 \leq n < 2^r } |P(\tau_{\ell,r}(n+1)) - P(\tau_{\ell,r}(n))| \leq \lambda_r
\]
for every $\ell \geq \ell_0$ and $r\ge 1$. It thus follows from 
Proposition \ref{prop:dyadic} that for almost all $f$ there is $\ell_0$ depending on $f$ such that
\begin{align*}
\max_{\sigma \in [\sigma_{\ell}, \sigma_{\ell-1}]}|P(\sigma) - P(\sigma_{\ell})| & \leq 2 \sum_{r = 1}^{\infty} \lambda_r 
                           = 2\sqrt{2 C_1} \sum_{r = 1}^{\infty} \frac{\sqrt{r}}{2^r} = C < \infty
\end{align*}
for every $\ell\ge \ell_0$. Here, $C$ is an absolute constant. So, the assertion in the lemma follows.
\end{proof}

\section{Proof of Proposition \ref{prop:lowerboundresult}}\label{sec:harper}

We begin by understanding the size of the random variables of the form
\[
\sup_{1 \leq t \leq T_k} \sum_p \frac{f(p)\cos(t \log p)}{p^{\sigma_k}}
\]
for carefully chosen sequences $T_k$ and $\sigma_k$. To this end, we rely on a result by Harper \cite{harper2013}, which is stated here as Corollary~\ref{cor:harper_result}. Since the main ingredient in the proof of this corollary
is one of Harper's results \cite[Corollary 2]{harper2013} on Gaussian processes, we state that here as well. 

\begin{theorem}[Corollary 2 in \cite{harper2013}]\label{harpertheorem}
    Let $\{g_p\}$ be a sequence of standard normal Gaussian random variables indexed by the primes. 
    As $x \to \infty$,
    \[
    \mathbb{P} \left( \sup_{1 \leq t \leq 2(\log \log x)^2} \sum_{p \leq x} g_p \frac{\cos (t \log p)}{p^{1/2 + 1/\log x}} \leq \log \log x - \log \log \log x + O \Big( (\log \log \log x)^{3/4} \Big) \right)
    \]
    is $O\Big( (\log \log \log x)^{-1/2} \Big).$
\end{theorem}

\begin{cor}[Analysis from end of \S{6.3}, \cite{harper2013}]\label{cor:harper_result}
    Let $f$ be a random multiplicative function. Let $y = \log^8 x$. As $x \to \infty$, 
    \[
    \mathbb{P} \left( \sup_{1 \leq t \leq 2(\log \log x)^2} \sum_{p} \frac{f(p)\cos(t \log p)}{p^{1/2 + 1/\log x}} \leq \log \log x - \log \log y - O(1) - ( \log \log \log x)^{3/4} \right) 
    \]
    is $O\Big( (\log \log \log x)^{-1/2} \Big).$
\end{cor}

\begin{proof}
As explained at the end of \S{6.3} in \cite{harper2013}, 
this follows from a version of Theorem~\ref{harpertheorem} 
derived using a multivariate central limit theorem. 
This is detailed in \cite[Appendix B]{harper2013}, and yields a version 
 of Theorem~\ref{harpertheorem}
with the $g_p$ replaced with $f(p)$, where $f$ is a random multiplicative function. 
We remark that the choice $y=\log^8 x$ is as in \cite[p.\ 606]{harper2013}.
\end{proof}
We apply Corollary~\ref{cor:harper_result} along the sequence $x_k$ tending to infinity given in \eqref{eq:xkdefn}.

\begin{prop}[Explicit version of Corollary \ref{cor:harper_result}]
\label{prop:independentlowerbound}
    Let $c$ and $C_1$ be any constants satisfying $c > 2$ and $C_1 > 1$. Let
    \[
    \sigma_k := \frac{1}{2} + \frac{1}{\exp \left( \exp(k^c) \right)}.
    \]
   Consider the event $E_k:= E_k(c, C_1)$ defined as
    \begin{multline*}
    \Bigg[ \sup_{1 \leq t \leq 2\log^2\left( \frac{1}{\sigma_k - \frac{1}{2}} \right) } \sum_p \frac{f(p) \cos(t \log p)}{p^{\sigma_k}} \geq \log \left( \frac{1}{\sigma_k - 1/2} \right) - C_1 \log \log \left( \frac{1}{\sigma_k - 1/2} \right)  \Bigg].
     \end{multline*}
    Then
    \[
    \mathbb{P} \left( \liminf_{k \to \infty} E_k \right) = 1.
    \]
    Consequently, for almost all random multiplicative functions $f$ there exists a constant $k_0 := k_0(f, c, C_1) \geq 1$ such that
    \[
    \sup_{1 \leq t \leq 2\log^2\left( \frac{1}{\sigma_k - \frac{1}{2}} \right) } \sum_p \frac{f(p) \cos(t \log p)}{p^{\sigma_k}} \geq \log \left( \frac{1}{\sigma_k - 1/2} \right) - C_1 \log \log \left( \frac{1}{\sigma_k - 1/2} \right)
    \]
    for each $k \geq k_0.$
\end{prop}

\begin{proof}
   Let 
   \begin{equation}\label{eq:xkdefn}
   \log x_k := \frac{1}{\sigma_k - 1/2} = \exp\left( \exp(k^c) \right).
   \end{equation}
   Consider the events $\mathcal{A}_k$ defined for $k \geq 1$ as
    \begin{multline*}
    \Bigg[ \sup_{1 \leq t \leq 2(\log \log x_k)^2} \sum_{p} \frac{f(p)\cos(t \log p)}{p^{1/2 + 1/\log x_k}} \leq \log \log x_k - \log \log (\log^8 x_k) \\
    - O(1) - ( \log \log \log x_k)^{3/4} \Bigg].
    \end{multline*}
    Since $x_k \to \infty$ as $k \to \infty$, and $C>1$, we can find $k_0 := k_0(c, C_1) \geq 1$ so that 
    for each $k \geq k_0$,
   \[
   \log \log x_k - C_1 \log \log \log x_k \leq \log \log x_k - \log \log ( \log^8 x_k ) - O(1) - (\log \log \log x_k)^{3/4}.
   \]
  In other words, we have $\overline{E_k} \subseteq \mathcal{A}_k$ for each $k \geq k_0$, where $\overline{E_k}$ denotes the complement of $E_k$. Therefore,
  \[
   \sum_{k \geq 1} \mathbb{P} (\overline{E_k}) \ll \sum_{k \geq 1} \mathbb{P} (A_k).
  \]
  Furthermore, an application of Corollary \ref{cor:harper_result} provides 
   \[
   \mathbb{P}(A_k) = O\Big( (\log \log \log x_k)^{-1/2} \Big) = O \Big( k^{-c/2} \Big)
   \]
   as $k \to \infty$.
   So, we have 
   \begin{align*}
 \sum_{k \geq 1} \mathbb{P} (\overline{E_k}) & \ll \sum_{k \geq 1} \mathbb{P} (A_k) \ll \sum_{k \geq 1} k^{-c/2} < \infty  
   \end{align*}
   since $c > 2$. By the Borel-Cantelli lemma, the argument is complete.


\end{proof}

\begin{proof}[Proof of Proposition \ref{prop:lowerboundresult}]
     By \cite{wintner1944}, $F(\sigma + it)$ has an Euler product whenever $\sigma > 1/2$, almost surely. So, since $\sigma_k>1/2$, we have almost surely
    \begin{equation*}
    \begin{split}
     |F(\sigma_k + it)| & =  \exp\Big( \Re \log F(\sigma_k + it) \Big) \\
     & = \exp \left( \Re \left( \sum_{p} \frac{f(p)}{p^{\sigma_k + it}} - \frac{1}{2} \sum_p \frac{1}{p^{2\sigma_k + 2it}} + \sum_{m = 3}^{\infty} \sum_{p} \frac{(-1)^{m+1} f(p)^m}{mp^{m(\sigma_k + it)}} \right) \right)\\
    & = \exp\Bigg( \sum_{p}\frac{f(p) \cos(t \log p)}{p^{\sigma_k}} - \frac{1}{2} \sum_p \frac{\cos(2t \log p)}{p^{2\sigma_k}} \\
    & \hspace{4.5cm} + \sum_{m = 3}^{\infty} \sum_p \frac{(-1)^{m+1} f(p)^m \cos(mt \log p)}{m p^{m\sigma_k}} \Bigg)
    \end{split}
    \end{equation*}
    for any $k \geq 1$ and $t \in \mathbb{R}$ after log expanding. Here, the logarithm branch refers to one where $\lim_{\sigma\to +\infty} \log F(\sigma)=0$.
In view of this, we can find an almost sure lower bound for
\[
\sup_{1 \leq t \leq \log^2\left( \frac{1}{\sigma_k - 1/2} \right)} |F(\sigma_k + it)|
\]
by finding an almost sure lower bound for
\begin{multline}\label{eq:mainsup}
\sup_{1 \leq t \leq \log^2\left( \frac{1}{\sigma_k - 1/2} \right)} \sum_{p}\frac{f(p) \cos(t \log p)}{p^{\sigma_k}} - \frac{1}{2} \sum_p \frac{\cos(2t \log p)}{p^{2\sigma_k}} \\ + \sum_{m = 3}^{\infty} \sum_p \frac{(-1)^{m+1} f(p)^m \cos(mt \log p)}{m p^{m\sigma_k}}.
\end{multline}
To do this, we will use the following observation: for any real-valued functions $\phi, \psi, \rho$ defined on a closed bounded interval $I \subset \mathbb{R}$, we have that
\[
\sup_{x \in I} \Big(\phi(x) - \psi(x) + \rho(x)\Big) \geq \left(\sup_{x \in I} \phi(x) \right) - \left( \sup_{x \in I} \psi(x) \right) + \left(\inf_{x \in I} \rho(x) \right).
\]
Specifically, we will find a lower bound for the expressions
\begin{equation}\label{eq:primary_term}
    \sup_{1 \leq t \leq \log^2 \left( \frac{1}{\sigma_k - 1/2} \right)} \sum_{p}\frac{f(p) \cos(t \log p)}{p^{\sigma}}
\end{equation}
and
\begin{equation}\label{eq:constant_term}
    \inf_{1 \leq t \leq \log^2\left( \frac{1}{\sigma_k - 1/2} \right)} \sum_{m = 3}^{\infty} \sum_p \frac{(-1)^{m+1} f(p)^m \cos(mt \log p)}{m p^{m\sigma}}
\end{equation}
and an upper bound for
\begin{equation}\label{eq:secondary_term}
    \sup_{1 \leq t \leq \log^2 \left( \frac{1}{\sigma_k - 1/2} \right)} \frac{1}{2} \sum_p \frac{\cos(2t \log p)}{p^{2\sigma_k}}.
\end{equation}

First by Proposition~\ref{prop:independentlowerbound}, for almost all $f$ there exists $k_1:= k_1(f, c, C_1) \geq 1$ such that 
\[
\sup_{1 \leq t \leq \log^2 \left( \frac{1}{\sigma_k - 1/2} \right)} \sum_{p}\frac{f(p) \cos(t \log p)}{p^{\sigma}} \geq \log \left( \frac{1}{\sigma_k - 1/2} \right) - C_1 \log \log \left( \frac{1}{\sigma_k - 1/2} \right) 
\]
for each $k \geq k_1$. Therefore, we have an eventual lower bound on (\ref{eq:primary_term}).

Second, for any $t \geq 1$, we have
    \begin{align*}
    \left| \sum_{m = 3}^{\infty} \sum_p \frac{(-1)^{m+1} f(p)^m \cos(mt \log p)}{m p^{m\sigma_k}} \right| & \leq \sum_{m = 3}^{\infty} \sum_p \frac{1}{m p^{m \sigma_k}} \\
    & \leq \sum_{m = 3}^{\infty} \sum_p \frac{1}{p^{m/2}} \\
    & = \sum_{p} \sum_{m = 3}^{\infty} \left( \frac{1}{p^{1/2}} \right)^m \\
    & = \sum_p \frac{1}{p ( p^{1/2} - 1)} \leq 2.112.
    \end{align*}
    To obtain the last inequality, we computed the sum over the first 9 million primes, estimated the tail by an integral, and rounded-up the result. As the above bound holds for any $t \geq 1$, the sum (\ref{eq:constant_term}) is $\ge -2.112$, independent of $k$.

And third, for any $t \geq 1$,
    \[
    \left| \sum_p \frac{\cos(2t \log p)}{p^{2\sigma_k}} \right| \leq \sum_{p} \frac{1}{p^{2\sigma_k}}.
    \]
    So,
    \[
    \left| \sup_{1 \leq t \leq \log^2 \left( \frac{1}{\sigma_k - 1/2} \right)} \frac{1}{2} \sum_p \frac{\cos(2t \log p)}{p^{2\sigma_k}} \right| \leq \frac{1}{2} \sum_{p} \frac{1}{p^{2\sigma_k}}.
    \]
    By Proposition \ref{prop:zetaasym},
    \[
    \sum_{p} \frac{1}{p^{2\sigma_k}} \sim \log \left( \frac{1}{2\sigma_k - 1} \right) 
    \]
   as $2\sigma_k \to 1^+$. Therefore, letting $C$ be any constant satisfying
   $C > 1/2$, there exists $k_2:= k_2(C)$ such that
    \[
    \frac{1}{2}  \sum_p \frac{ 1}{p^{2\sigma_k}} \leq C\log \left( \frac{1}{2\sigma_k - 1} \right) = C \log \left( \frac{1}{\sigma_k - 1/2} \right) - C \log 2
    \]
    for each $k \geq k_2$. Thus, expression (\ref{eq:secondary_term}) is bounded above by \[
    C \log \left( \frac{1}{\sigma_k - 1/2} \right) - C \log 2.
    \]
    
Combining the bounds for (\ref{eq:primary_term}), (\ref{eq:constant_term}), and (\ref{eq:secondary_term}), it follows that for any $C_1 > 1$ and $C > 1/2$, and for almost all $f$, there exists $k_0 := \max(k_1, k_2) \geq 1$ such that
\begin{align*}
    (\ref{eq:mainsup}) & \geq \log \left( \frac{1}{\sigma_k - 1/2} \right) - C_1 \log \log \left( \frac{1}{\sigma_k - 1/2} \right) - C \log \left( \frac{1}{\sigma_k - 1/2} \right) + C \log 2 - 2.112
\end{align*}
for each $k \geq k_0$.
Letting $C_0:= 1 - C$ and $C_2 := C\log 2 - 2.112$, and additionally requiring that $C<1$ so that $C_0$ is positive, 
the right-side of the last inequality can be 
written as 
\[
C_0 \log \left( \frac{1}{\sigma_k - 1/2} \right) - C_1 \log \log \left( \frac{1}{\sigma_k - 1/2} \right) + C_2,
\]
where $C_0 \in (0, 1/2)$ and $C_2 \in (-1.765, -1.419)$.
\end{proof}

\section{Acknowledgements}


NG would like to thank the RTG grant NSF DMS-2231565 for necessary financial support. Furthermore, NG thanks GH for his invaluable patience and enlightening discussions related to this work.

\appendix

\section{Useful Analysis Results}

\begin{defn}
Given an interval $[a, b]$, where $-\infty < a < b < \infty$, we denote the \emph{$r$-th decomposition} of $[a, b]$ by
\[
D_r:=D_r([a, b]) = \{ \tau_{r}(n) : 0 \leq n \leq 2^{r} \},
\]
where, for positive integers $r$,
\[
\tau_{r}(n) := a + \frac{n}{2^{r}}( b - a).
\]
\end{defn}

\begin{prop}\label{prop:dyadic}
Let $[a, b]$ be an interval and let $D_{r}$ be its $r$-th decomposition. Let $f$ be a continuous function on $[a, b]$ and suppose that for all $r \geq 1$ there are numbers $\lambda_{r}$, indexed by the positive integers $r$, such that
    \[
    \max_{0 \leq n < 2^{r}} |f(\tau_{r}(n+1)) - f(\tau_{r}(n))| \leq \lambda_{r}.
    \]
    Then, for any $s, t \in [a, b]$ we have
    \[
    |f(s) - f(t)| \leq 2\sum_{r = R + 1}^{\infty} \lambda_{r},
    \]
    where $R$ is the integer satisfying
    \[
    \frac{|b - a|}{2^{R+1}} < |s - t| \leq \frac{|b - a|}{2^{R}}.
    \]
\end{prop}

\begin{proof}
    Take any $s, t \in [a, b]$ with $s < t$. For integers $r >R$, set
    \begin{equation*}\label{eq:maxbound}
    s_{r} := \max \left\{ u \leq s : u \in D_{r} \right\}\qquad \text{and}\qquad
     t_{r} := \min \left\{ u \geq t : u \in D_{r} \right\}.
    \end{equation*}
    Notice that $s_{r} \to s$ from below and $t_{r} \to t$ from above as $r \to \infty$. Since $f$ is continuous on $[a, b]$, we have
    \[
    |f(s) - f(t)| = \lim_{r \to \infty} |f(s_{r}) - f(t_{r})|.
    \]
    We will show by induction that
    \begin{equation}\label{induction}
         |f(s_M) - f(t_M)| \leq 2 \sum_{r= R+1}^M \lambda_{r}
    \end{equation}
    for all integers $M > R$.
    
    Suppose that $M = R + 1$. Notice that the constraint on $R$ forces $s_R = \tau_{R}(n)$ and
    $t_R = \tau_{R}(n+1)$ for some $0 \leq n \leq 2^R.$ Additionally, note that 
    \begin{equation}\label{eq:tauidentity}
    \tau_{R+1}(2n) = \tau_R(n).
    \end{equation}
    In view of relation (\ref{eq:tauidentity}), we have 
    \[
    \tau_{R+1}(2n) \leq s < \tau_{R+1}(2n+1) < t \leq \tau_{R+1}(2n+2).
    \]
    For otherwise (if $t\le \tau_{R+1}(2n+1)$) we get $|s - t| \leq |b-a|/2^{R+1}$, contradicting the choice of $R$.
    Hence, 
    \[
    s_{R+1} = \tau_{R+1}(2n) \qquad \text{and} \qquad t_{R+1} = \tau_{R+1}(2n+2).
    \]
    Using the hypothesis in the proposition, it follows that
    \begin{align*}
        |f(s_{R+1}) - f(t_{R+1})| & \leq |f(s_{R+1}) - f(\tau_{R+1}(2n+1))| + |f(\tau_{R+1}(2n+1)) - f(t_{R+1})| \\
                                  & \leq 2 \lambda_{R+1}.
    \end{align*}   
    Hence, the base case is verified.

    Now, let $M$ be an integer such that $M > R + 1$ and
    \begin{equation}\label{eq:inductivehypo}
    |f(s_M) - f(t_M)| \leq 2 \sum_{r = R+1}^M \lambda_{r}.
    \end{equation}
    We will verify \eqref{induction} for $M+1$. Notice that if $s_{M} = \tau_M(n)$ for some $n$ then by the definition of $s_M$ we have
    \[
    \tau_M(n) \leq s < \tau_M(n+1).
    \]
    By the formula \eqref{eq:tauidentity}, this implies
    \[
    s_M = \tau_{M+1}(2n) \leq s < \tau_{M+1}(2n + 2)
    \]
    which results with two candidates for $s_{M+1}$: either $\tau_{M+1}(2n)$ or $\tau_{M+1}(2n+1).$ In either case, we get
    \[
    |s_{M+1} - s_M| \leq \frac{|b-a|}{2^{M+1}}.
    \]
    A similar argument can be done to show that $|t_{M+1} - t_M|$ has the same upper bound. Consequently, by the hypothesis in the proposition,
    \begin{equation}\label{eq:lambdabounds}
    |f(s_{M+1}) - f(s_{M})| \leq \lambda_{M+1} \qquad \text{and}\qquad |f(t_{M+1}) - f(t_{M})| \leq \lambda_{M+1}
    \end{equation}
    as $s_{M+1}$ and $s_M$ (similarly $t_{M+1}$ and $t_M$) are either identical or consecutive in $D_{M+1}.$
    So, on writing
    \[
    |f(s_{M+1}) - f(t_{M+1})| \leq |f(s_{M+1}) - f(s_{M})| +|f(s_{M}) - f(t_{M})|+ |f(t_{M+1}) - f(t_{M})|,
    \]
    and applying the bounds from (\ref{eq:lambdabounds}) and the inductive hypothesis (\ref{eq:inductivehypo}), we get
    \[
    |f(s_{M+1}) - f(t_{M+1})| \leq \lambda_{M+1} + \lambda_{M+1} + 2 \sum_{r > R}^M \lambda_{r} = 2\sum_{r > R}^{M+1} \lambda_r,
    \]
    which concludes the proof.

\end{proof}

\begin{prop}\label{prop:zetaasym}
    If $x \to 1^+$, then
    \[
    \sum_p \frac{1}{p^x} \sim \log \zeta(x) \sim \log \left( \frac{1}{x - 1} \right).
    \]
\end{prop}

\begin{proof}
For $x > 1$, we may write $\zeta(x)$ as an Euler product and log-expand to get
\begin{align*}
\log \zeta(x) & = \sum_p \sum_{n \geq 1} \frac{1}{n p^{nx}} \\
    & = \sum_p \frac{1}{p^x} + \sum_{n \geq 2} \sum_p \frac{1}{n p^{nx}} \\
    & = \sum_p \frac{1}{p^x} + O(1).
\end{align*}
Since the prime sum tends to $\infty$ as $x\to 1^+$, the first asymptotic follows.

For the second asymptotic, we use the classical result in number theory that
\[
\zeta(x) \sim \frac{1}{x - 1}
\]
as $x \to 1^+.$ Since logarithms preserve asymptotics, we have that
\[
\log \zeta(x) \sim \log \left( \frac{1}{x - 1}\right) 
\]
as $x \to 1^+$.
    
\end{proof}

\begin{lemma}\label{lem:claim1}
    There is an absolute positive constant $C$ such that if $1/2 < \sigma \leq 1$, then
    \[
    \sum_{p} \frac{\log^2 p}{p^{2 \sigma}} \leq C \frac{1}{(2\sigma - 1)^2}.
    \]
\end{lemma}

\begin{proof}

Take any $1/2 < \sigma \leq 1$.
Since $2\sigma > 1$, the series in question converges. 
So, we may express the series as the Riemann-Stieltjes integral
\begin{align}
\sum_p \frac{(\log p)^2}{p^{2\sigma}} & = \int_{1}^{\infty} \frac{(\log x)^2}{x^{2\sigma}} \ d\pi(x), \notag
\end{align}
where $\pi(x)$ denotes the prime counting function. 
As the function $(\log x)^2 x^{-2\sigma}$ is continuous on $(1, \infty)$, 
we may apply integration by parts formula (e.g.\ \cite[Theorem A.2]{montgomery}) to get
\begin{equation}\label{eq:reexpression}
\int_{1}^{\infty} \frac{(\log x)^2}{x^{2\sigma}} \ d\pi(x) = \left. \frac{\pi(x) (\log x)^2}{x^{2 \sigma}} \right|_ 1^\infty - \int_{1}^{\infty} \pi(x) \frac{2 \log x - 2\sigma (\log x)^2}{x^{2\sigma + 1}} \ dx.
\end{equation}
Using $\pi(x) \sim x/\log x$ as $x \to \infty$ by the Prime Number Theorem, gives
\begin{align*}
\left. \frac{\pi(x) (\log x)^2}{x^{2 \sigma}} \right|_1^\infty & = \lim_{N \to \infty}\left.\frac{\pi(x) (\log x)^2}{x^{2 \sigma}} \right|_1^N \\
            & = \lim_{N \to \infty} \frac{\pi(N) (\log N)^2}{N^{2 \sigma}} \\
            & = \lim_{N \to \infty} \frac{\log N}{N^{2\sigma - 1}}
\end{align*}
Since $2\sigma - 1 > 0$, we therefore obtain
\[
\left. \frac{\pi(x) (\log x)^2}{x^{2 \sigma}} \right|_1^\infty = 0.
\]
Thus, going back to (\ref{eq:reexpression}), on absorbing the negative sign, we have
\begin{align}
\int_{1}^{\infty} \frac{(\log x)^2}{x^{2\sigma}} \ d\pi(x) & = \int_1^{\infty} \pi(x) \frac{2\sigma(\log x)^2 - 2 \log x}{x^{2\sigma + 1}} \ dx \label{eq:intform}
\end{align}
Note that $\pi(x) = 0$ for $x \in [1, 2)$ and that $\pi(x) < 2 x/\log x$ for $x \geq 2$ by \cite[Corollary 2]{approx_pnt}. So, (\ref{eq:intform}) can be bounded above by
\[
< 2 \int_2^{\infty} \frac{2\sigma\log x - 2}{x^{2\sigma}} \ dx.
\]
Since $2\sigma \log x - 2 < 0$ when $x \leq e^{1/\sigma}$, we can further bound above
for any $\sigma$ satisfying $1/2 < \sigma \leq 1$ by starting our integration at $e^2.$ So,
\begin{align*}
    \int_{2}^{\infty} \frac{(\log x)^2}{x^{2\sigma}} \ d\pi(x) & < 2 \int_{e^2}^{\infty} \frac{2\sigma\log x - 2}{x^{2\sigma}} \ dx \\
    & = \frac{2e^{2 - 4\sigma}(1 - 3\sigma + 4\sigma^2)}{(2\sigma - 1)^2} \\
    & \leq \frac{4}{(2\sigma - 1)^2}.
\end{align*}
    This provides the desired bound with $C = 4.$
\end{proof}

\section{Useful Probability Results}

\begin{prop}[Hoeffding's Inequality]
\label{thm:hoeffding}
    Let $\{X_i\}_{n \in \mathbb{N}}$ be a sequence of i.i.d. random variables such that $a_n \leq X_n \leq b_n$ almost surely. Then for any $\lambda > 0$ we have
    \[
    \mathbb{P} \left( \sum_{n \geq 1} X_n - \sum_{n \geq 1} \mathbb{E}[X_n] \geq \lambda \right) \leq \exp \left( - \frac{2\lambda^2}{\sum_{n \geq 1} (b_n - a_n)^2} \right).
    \]
\end{prop}

\begin{proof}
    Refer to the proof of Lemma 2.3 in \cite{LIL}. 
\end{proof}

\begin{prop}[Borel-Cantelli Lemma]\label{prop:borel}
    Let $(A_1, A_2, \ldots)$ be a sequence of events in a common probability space $(\Omega, \mathcal{F}, \mathbb{P})$. If $\sum_{n \geq 1} \mathbb{P}(A_n) < \infty$ then
    \[
    \mathbb{P} \left( \limsup_{n \to \infty} A_n \right) = 0.
    \]
\end{prop}

\begin{proof}
    Refer to the proof of Lemma 3 in Chapter 6 of \cite{probbook}.
\end{proof}

\begin{prop}[Kolmogorov Three-Series Theorem]\label{prop:twoseries}
    Let $(X_1, X_2, \ldots)$ be an independent sequence of $\mathbb{R}$-valued random variables, and let $b$ be a positive real number. Define
    \[
    Y_n(\omega) = 
    \begin{cases}
        X_n(\omega) & \text{if } |X_n(\omega)| \leq b \\
        0 & \text{otherwise.}
    \end{cases}
    \]
    Then $\sum_n X_n$ converges almost surely if and only if the following three series converge:
    \begin{enumerate}[label={(\alph*)}]
        \item 
        \[
        \sum_{n} \mathbb{P} \Big( [X_n \neq Y_n] \Big);
        \]
        \item 
        \[
        \sum_{n} \mathbb{E} [Y_n];
        \]
        \item 
        \[
        \sum_{n} \var (Y_n).
        \]
    \end{enumerate}
    If one of these three series diverges, then $\sum_n X_n$ diverges almost surely.
\end{prop}

\begin{proof}
    Refer to the proof of Theorem 27 in Chapter 12 of \cite{probbook}.
\end{proof}

\newpage

\printbibliography

\end{document}